\documentclass[a4paper,12pt]{article}
\usepackage{amsmath, amsfonts, amssymb, amsthm}
\usepackage[english]{babel}

\newcounter{num}[section]

\newenvironment{theorem}
{\refstepcounter{num}%
\bigskip\noindent\nopagebreak[4]{\bf Theorem~\arabic{section}.\arabic{num}. }\it}
\newenvironment{lemma}
{\refstepcounter{num}%
\bigskip\noindent\nopagebreak[4]{\bf Lemma~\arabic{section}.\arabic{num}. }\it}
\newenvironment{remark}
{\refstepcounter{num}%
\bigskip\noindent\nopagebreak[4]{\bf Remark~\arabic{section}.\arabic{num}. }}

\newcommand{\N}{{\mathbb{N}}}

\newcommand{\LL}{{\mathcal{L}}}

\newcommand{\V}{{\mathrm{V}}}

\newcommand{\om}{{\omega}}
\newcommand{\pr}{{\prime}}

\newcommand{\s}{{\sigma}}

\newcommand{\al}{{\alpha}}
\newcommand{\be}{{\beta}}

\newcommand{\Tcal}{{\mathcal{T}}}

\renewcommand{\a}{{\mathbf{a}}}

\newcommand{\pbf}{{\mathbf{p}}}
\newcommand{\qbf}{{\mathbf{q}}}
\newcommand{\rbf}{{\mathbf{r}}}

\newcommand{\Pbf}{{\mathbf{P}}}
\newcommand{\Qbf}{{\mathbf{Q}}}
\newcommand{\Rbf}{{\mathbf{R}}}

\newcommand{\one}{{\mathbf{1}}}
\newcommand{\zero}{{\mathbf{0}}}

\newcommand{\qq}{{\mathbf{q}_\om}}
\newcommand{\At}{{\mathcal{T}}}

\begin{document}

\title{Wreath products of semigroups in universal algebraic geometry and Plotkin`s problem}
\author{A.N. Shevlyakov}

\maketitle

\abstract{We prove that the wreath product $C=A\wr B$ of a semigroup $A$ with zero and an infinite cyclic semigroup $B$ is $\qq$-compact (logically Noetherian). Our result partially solves the Plotkin`s problem about wreath products in~\cite{plotkin}.}

\section{Introduction}

Our research is devoted to the studies in universal algebraic geometry. This direction was initiated in papers by B.~Plotkin~\cite{plotkin_lectures} and E.~Daniyarova, A.~Miasnikov, V.~Remeslennikov~\cite{DMR2}, and it deals with equations over various algebraic structures (semigroups, groups, Lie algebras, etc.). In the current paper we study semigroup equations, but the problem we solved comes from group theory.

In our paper we deal with wreath products. Wreath products plays a central role in the studies of group varieties and extensions of groups. The definition of a wreath product may be naturally given for semigroups, where wreath products have important applications in Krohn-Rhodes theory. 

Algebraic geometry (studying of equations) over wreath products was developed in~\cite{BMRom}. Namely, in~\cite{BMRom} it was proved that the wreath product $C=A\wr B$ of a non-abelian group $A$ and infinite $B$ is not equationally Noetherian, i.e. there exists an infinite system of equations $S$ over $C$ such that $S$ is not equivalent to any finite subsystem. In~\cite{shevl_shah} we obtained the same result using another approach. 

In~\cite{plotkin} B.~Plotkin gave the following list of problems for wreath products.
\begin{enumerate}
\item When a wreath product $C=A\wr B$ is equationally Noetherian?
\item When $C$ is $\qq$-compact but not equationally Noetherian?
\item When $C$ is not $\qq$-compact?

\end{enumerate}

First, we note that in the B.~Plotkin`s papers this notion is called {\it logically Noetherian property} (also B.Plotkin used the term ``geometrically Noetherian'' instead of ``equationally Noetherian''). However, in our paper we will use the terminology of~\cite{DMR2}. 

Let us explain the notion of the $\qq$-compactness. The idea of $\qq$-compactness plays a similar role in universal algebraic geometry as Compactness theorem in model theory. If any solution $P\in A^n$ of a system of equations $S(X)$ also satisfies an equation $E(X)$, then $E(X)$ is said to be an consequence of $S(X)$ over an algebraic structure (group, semigroup, etc.) $A$. If $S(X)$ is infinite there arises a question: is there a finite subsystem of $S(X)$ with the consequence $E(X)$? If an algebraic structure $A$ admits a positive answer for the last question, then $A$ is called $\qq$-compact (the formal and complete definition of $\qq$-compactness is given in our paper below). According to the definition, any equationally Noetherian algebraic structure is $\qq$-compact (this inclusion explains the statement of the second Plotkin`s question above). 

There are few examples of $\qq$-compact algebraic structures that are not equationally Noetherian. In~\cite{plotkin_monster} B.~Plotkin proved that the direct product of all finitely generated groups is  $\qq$-compact, but not equationally Noetherian. In~\cite{shevl_seams} there were studied semigroup equations with constants and defined $\qq$-compact semilattices. 

In our paper we study equations over wreath products of semigroups. All studied equations are coefficient-free, i.e. they have not occurrences of constants, just variables. For such class of equations we prove  two theorems. The following theorem is devoted to the first Plotkin`s question.

\medskip

\noindent{\bf Theorem A.} {\it Suppose a semigroup $A$ contains zero and  a semigroup $B$ is infinite cyclic. The wreath product $C=A\wr B$ is equationally Noetherian iff $A$ is nilpotent.
}

\medskip

Our second theorem deals with the second Plotkin`s question.

\medskip

\noindent{\bf Theorem B.}
The wreath product $C=A\wr B$ of two semigroups $A,B$ is $\qq$-compact if $A$ contains a zero and $B$ is infinite cyclic. 

\medskip

According to our theorems, there exist many examples of wreath products which are not equationally Noetherian but $\qq$-compact. Moreover, wreath products have an interesting property: a semigroup $A$ may be not $\qq$-compact, but nevertheless the whole wreath product $A\wr B$ is $\qq$-compact.

\section{Preliminaries}

Let $\LL=\{\cdot\}$ be a semigroup language. Commutative semigroup are often considered as algebraic structures of the additive language $\LL_+=\{+\}$. The set of all terms of the language $\LL$ ($\LL_+$) is denoted by $\Tcal$ ($\Tcal_+$). For shortness, such sets are called $\Tcal$-terms and $\Tcal_+$-terms respectively. Clearly, any $\Tcal$-term in variables $X=\{x_1,\ldots,x_n\}$ can be written as a product 
\begin{equation*}
x_{i_1}x_{i_2}\ldots x_{i_k} \; (1\leq i_j\leq n).
\end{equation*} 
Similarly, any $\Tcal_+$-term is equivalent to an expression of the form
\begin{equation*}
\al_1 x_1+\al_2 x_2+\ldots+\al_n x_n \; (\al_i\in\N=\{0,1,2,\ldots\})
\end{equation*} 
over any commutative semigroup.

A {\it semigroup equation} (equation, for shortness) is an equality
\[
t(X)=s(X), \; t,s\in\Tcal.
\]
In other words, any semigroup equation is an atomic formula of the language $\LL$.
Similarly, an {\it additive equation} is an atomic formula $t(X)=s(X)$, $t,s\in\Tcal_+$ 
of the language $\LL_+$.

Below we sometimes denote equations by capital letters, e.g. $E(X)$. The denotation $E(X)\colon t(X)=s(X)$ means that an equation $E(X)$ consists of two parts $t(X)$ and $s(X)$. 

Let us give the main definitions of algebraic geometry over semigroups (see~\cite{DMR2} for more details). All definitions below are given for the language $\LL$, but one can easily adapt them for the additive language $\LL_+$.

A system of semigroup equations (system, for shortness) $S(X)$ is an arbitrary set of semigroup equations. {\it Notice that we always consider  systems with a finite set of variables $X=\{x_1,\ldots,x_n\}$}. The set of all solutions of $S(X)$ in a semigroup $A$ is denoted by $\V_A(S)\subseteq A^n$. Two systems $S_1,S_2$ are called {\it equivalent over a semigroup $A$} if $\V_A(S_1)=\V_A(S_2)$.

A semigroup $A$ is {\it equationally Noetherian} if any infinite system $S(X)$ is equivalent to a finite subsystem $S^\ast(X)\subseteq S(X)$ over $A$. A  semigroup $A$ is {\it $\qq$-compact} if for any infinite system $S(X)$ and an equation $E(X)$ such that 
\begin{equation}
\V_A(S)\subseteq \V_A(E)
\label{eq:inclusion1}
\end{equation}
there exists a finite subsystem $S^\ast(X) \subseteq S(X)$ with
\begin{equation}
\V_A(S^\ast)\subseteq \V_A(E).
\label{eq:inclusion2}
\end{equation}

According to the definitions, any equationally Noetherian semigroup is $\qq$-compact.

%\begin{remark}
%According to the Plotkin`s terminology, the $\qq$-compactness is called {\it logically Noetherian property}.
%\end{remark}

\bigskip

Let $A,B$ be semigroups. The \textit{direct power} $A^B=\prod_{b\in B}A$ of $A$ with the index set $B$ is the set of all tuples 
\[
(a_b\mid b\in B),\; a_b\in A
\]
indexed by elements of $B$.
The semigroup $A^B$ admits the coordinate-wise multiplication
\[
(a_b\mid b\in B)\cdot(a^\pr_b\mid b\in B)=(a_ba^\pr_b\mid b\in B).
\]
For any direct power $A^B$ one can define the $b$-th projection ($b\in B$) by 
\[
\pi_b((a_b\mid b\in B))=a_b.
\]

Let us give the central notion of this paper, the notion of the wreath product $A \wr B$ of two semigroups $A,B$. Since in our studies the second semigroup $B$ is always commutative, we treat $B$ below as a semigroup of the additive language $\LL_+$. 

The {\it wreath product } $C=A\wr B$ of two semigroups $A,B$ is a set of all pairs
\[
\{(\a,b)\mid \a\in A^B,b\in B\}.
\]
The multiplication in $A \wr B$ is defined as follows. Let $\a=(a_b\mid b\in B)\in A^B$, $\a^\pr=(a_b\mid b\in B)\in A^B$, then 
\begin{equation}
(\a,b_1)(\a^\pr,b_2)=(\a^{\pr\pr}, b_1+b_2),
\label{eq:mult_in_wreath}
\end{equation}
where $\a^{\pr\pr}=(a^{\pr\pr}_b\mid b\in B)$, $a^{\pr\pr}_b=a_ba^\pr_{b+b_1}$.

Below in our paper $B$ is the infinite cyclic semigroup. Obviously, $B$ is isomorphic to the additive semigroup of positive integers 
\[
B=\{1,2,3,\ldots\}.
\]
Also we will consider the additive monoid 
\[
B_0=\{0,1,2,\ldots\}.
\]
Denote
\begin{eqnarray*}
\Pi=A^B,&\; &\Pi_0=A^{B_0},\\
\Pi_1=A_1^B,&\; &\Pi_{01}=A_1^{B_0},\\
C=A\wr B,&\;& C_0=A\wr B_0,\\
&& C_{01}=A_1\wr B_0
\end{eqnarray*} 
where $A_1=A\cup \{1\}$ is the semigroup $A$ with the unit $1$ adjointed.

The difference between the semigroups $B$ and $B_0$ provides the following difference between $\Pi$ and $\Pi_0$: the first index of  $\a\in \Pi$ is $1$, but the first index of $\a^\pr\in\Pi_0$ is $0$. 

Let $\a=(a_1,a_2,\ldots)\in\Pi_1$, $\a^\pr=(a^\pr_0,a^\pr_1,\ldots)\in\Pi_{01}$. We write $\a\approx \a^\pr$ if $\a$ and $\a^\pr$ define the same infinite vector, i.e.
\[
a_i=a^\pr_{i-1}.
\] 

The multiplication in wreath products for cyclic semigroups $B,B_0$ has the following sense. Let $b\in B_0$ defines a shift map $\s_b\colon\Pi_{01}\to \Pi_{01}$,
\[
\s_b((a_0,a_1,\ldots))=(a_{b},a_{b+1},\ldots).
\]
Then the product~(\ref{eq:mult_in_wreath}) may be written as 
\begin{equation}
(\a_1,b_1)(\a_2,b_2)=(\a_1\cdot \s_{b_1}(\a_2), b_1+b_2),
\label{eq:mult_in_wreath_with_shift}
\end{equation}
where $\cdot$ is the coordinate-wise product of two vectors from $\Pi_{01}$. Below we will use the denotation~(\ref{eq:mult_in_wreath_with_shift}) for  multiplication in wreath products.

\section{Equations over $B$ and $B_0$}

Since the additive cyclic semigroup (monoid) $B$ ($B_0$) is an algebraic structure of the language $\LL_+$, any equation over $B,B_0$ is an equality
\[
\al_1 x_1+\al_2 x_2+\ldots+\al_n x_n =\be_1 x_1+\be_2 x_2+\ldots+\be_n x_n\; (\al_i,\be_i\in\N). 
\]
Such equations over $B_0$ were studied in the series of papers~\cite{shevl1,shevl2,shevl3}.

Let $S$ be a system of additive equations over $B$ and $Y=\V_B(S)$. The set $Y$ provides the following equivalence relation over $\Tcal_+$:
\[
t\sim s\Leftrightarrow t(P)=s(P)\mbox{ at each $P\in Y$.}
\]

Let $[t]$ denote the equivalence class of a term $t\in\Tcal_+$. The set of all equivalence classes is called the {\it coordinate semigroup }of the set $Y$ and it is denoted by $\Gamma_{B_0}(Y)$.

\begin{theorem}\textup{(\cite{shevl1})}
\label{th:N_discrimination}
The additive monoid $\Gamma_{B_0}(Y)$ is discriminated by $B_0$, i.e. for any finite subset $\{\gamma_1,\ldots,\gamma_n\}\subseteq M$ there exists a homomorphism of monoids $\psi\colon M\to B_0$ with $\psi(\gamma_i)\neq \psi(\gamma_j)$ for $i\neq j$. 
\end{theorem}

\begin{remark}
Actually, we proved that the monoid $B_0$ is a co-domain, i.e. any proper finite union of algebraic sets over $B_0$ is not algebraic. Following~\cite{DMR2}, any coordinate semigroup over a co-domain is discriminated by this co-domain. 
\end{remark}

\medskip

%One easily obtain the similar result for the semigroup $B$. 
%
%\begin{theorem}
%\label{th:N_discrimination}
%The coordinate semigroup $\Gamma_B(Y)$ is discriminated by $B$, i.e. for any finite subset $\{\gamma_1,\ldots,\gamma_n\}\subseteq \Gamma_B(Y)$ there exists a semigroup homomorphism $\psi\colon\Gamma_B(Y)\to B$ with $\psi(\gamma_i)\neq \psi(\gamma_j)$ for $i\neq j$. 
%\end{theorem}
%\begin{proof}
%Let $G=\{\gamma_1,\ldots,\gamma_n,0\}\subseteq B_0$. By Theorem~\ref{th:N_discrimination_old}, there exists a homomorphism of monoids $\psi\colon\Gamma_B(Y)\cup\{0\}\to B_0$ with $\psi(\gamma_i)\neq \psi(\gamma_j)$, $\psi(\gamma_i)\neq \psi(0)=0$ for $i\neq j$. Hence, the restriction $\psi|_B$ satisfies the claim of the theorem.
%\end{proof}

Let $S$ be a system  in variables $X=\{x_1,\ldots,x_n\}$ of additive equations over $B_0$ with the solution set $Y$. Any homomorphism $\psi\colon \Gamma_{B_0}(Y)\to B_0$ actually induces the substitution $[x_i]\mapsto b_i$. Hence, a homomorphism $\psi$ defines a point $Q=(b_1,\ldots,b_n)\subseteq B_0^n$ and one can directly prove that $Q\in \V_B(S)$. Thus, we may reformulate Theorem~\ref{th:N_discrimination} as follows.

\begin{theorem}
\label{th:N_discrimination_new}
Let $S$ be a system of additive equations over $B_0$. For a finite set of $\Tcal_+$-terms $t_1,\ldots,t_k$ there exists a point $Q\in\V_B(S)$ such that 
\[
t_i(Q)\neq t_j(Q) \mbox{ if $t_i\nsim t_j$}.
\]
\end{theorem}
 
The following simple result was also proved in~\cite{shevl1}.

\begin{theorem}
The infinite cyclic semigroup $B$ (monoid $B_0$) is equationally Noetherian.
\label{th:N_is_noeth}
\end{theorem}

\section{Terms and equations in wreath products}

The technique of the studies of equations over wreath products was found in~\cite{shevl_aut}, where we considered functional equations over direct powers of groups with occurrences of group automorphisms.

Since elements of the wreath product $C=A\wr B$ are pairs $(\a,b)$, $\a\in\Pi$, $b\in B$, any $\LL$-term $t(X)=x_{i_1}x_{i_2}\ldots x_{i_k}$ be an $\LL$-term is actually a pair of two terms $(t_A(X),t_B(X))$.  into  According to the multiplication in the wreath product $C=A\wr B$, the the term $t_A(X)$ is
\begin{equation}
t_A(X)=x_{i_1}\s_{x_{i_1}}(x_{i_2})\ldots \s_{x_{i_1}+\ldots+ x_{i_{k-1}}}(x_{i_k})
\label{eq:t_A(X)}
\end{equation}
and
\begin{equation}
t_B(X)=x_{i_1}+x_{i_2}+\ldots+x_{i_k}
\label{eq:t_B(X)}
\end{equation}
over $B$. 
 %Note that in~(\ref{eq:t_A(X)}) the lower indexes $y_i$ belong to $B$ and define shifts of elements (vectors) from $\Pi$.

Similarly, any equation $E(X)\colon t(X)=s(X)$ is equivalent over $C$ to the pair of equations $E_A(X),E_B(X)$:
\begin{eqnarray*}
t_A(X)=s_A(X),\\
t_B(X)=s_B(X),
\end{eqnarray*}
where the terms $t_A(X),t_B(X)$ are defined by~(\ref{eq:t_A(X)},\ref{eq:t_B(X)}) and 
\begin{eqnarray}
s_A(X)=x_{j_1}\s_{x_{j_1}}(x_{j_2})\ldots \s_{x_{j_1}+\ldots+ x_{j_{l-1}}}(x_{j_l}),
\label{eq:s_A(X)}\\
s_B(X)=x_{j_1}+x_{j_2}+\ldots+x_{j_l}
\label{eq:s_B(X)}
\end{eqnarray}
($x_{j_i}\in X$).

Below terms~(\ref{eq:t_A(X)},\ref{eq:s_A(X)}) and equation $t_A(X)=s_A(X)$ are called {\it wreath terms} and {\it wreath equation} respectively.

Denote points by
\begin{equation}
\label{eq:point_Pbf_and_P}
\Pbf=(\pbf_1,\ldots,\pbf_n),\; \pbf_i=(p_{i,1},p_{i,2},\ldots)\in\Pi,\; P=(p_1,\ldots,p_n)\in B^n.
\end{equation}

Thus, any system $S(X)=\{t^{(i)}(X)=s^{(i)}(X)\mid i\in I\}$ over $C$ is decomposed into two systems $S_A(X)$, $S_B(X)$ such that 
\begin{enumerate}
\item the system ${S}_A(X)=\{t^{(i)}_A(X)=s^{(i)}_A(X)\mid i\in I\}$ is a system of wreath equations over $\Pi$; 
\item the system ${S}_B(X)=\{t^{(i)}_B(X)=s^{(i)}_B(X)\mid i\in I\}$ is a system of additive equations over $B$.
\item if $\Pbf\in \V_{\Pi}(S_A)$ and $P\in\V_B(S_B)$ then $(\Pbf,P)\in\V_C(S)$.
\end{enumerate}

Let $t_A(X),s_A(X)$ be the wreath terms~(\ref{eq:t_A(X)},\ref{eq:s_A(X)}) and $P$ be the point from~(\ref{eq:point_Pbf_and_P}). We denote
\begin{eqnarray}
t_A(X,P)=x_{i_1}\s_{p_{i_1}}(x_{i_2})\ldots \s_{p_{i_1}+\ldots+ p_{i_{k-1}}}(x_{i_k})
\label{eq:t_A(X,P)},\\
s_A(X,P)=x_{j_1}\s_{p_{j_1}}(x_{j_2})\ldots \s_{p_{j_1}+\ldots+ p_{j_{l-1}}}(x_{j_l}).
\label{eq:s_A(X,P)}
\end{eqnarray}
The wreath equation $t_A(X)=s_A(X)$ becomes 
\begin{equation}
\label{eq:t_A(X,P)=s_A(X,P)}
t_A(X,P)=s_A(X,P).
\end{equation}
In~(\ref{eq:t_A(X,P)},\ref{eq:s_A(X,P)},\ref{eq:t_A(X,P)=s_A(X,P)}) all shifts $\s$ have constant values. Below such terms and equations are called {\it shift terms} and  {\it shift equations} respectively.

Shift terms are maps $t_A(X,P),s_A(X,P)\colon \Pi^n_{01}\to\Pi_{01}$. Hence, one can define the projections $\pi_b$ ($b\in B$) as follows:
\begin{eqnarray}
\label{eq:pi_b(t_A(X,P))}
\pi_b(t_A(X,P))=x_{i_1,b}x_{i_2,b+p_{i_1}}\ldots x_{i_k,b+p_{i_1}+\ldots+ p_{i_{k-1}}},\\
\label{eq:pi_b(s_A(X,P))}
\pi_b(s_A(X,P))=x_{j_1,b}x_{j_2,b+p_{j_1}}\ldots x_{j_l,b+p_{j_1}+\ldots+ p_{j_{l-1}}},
\end{eqnarray}
where each $x_i\in X$ is the vector $(x_{i,1},x_{i,2},\ldots)$. Obviously,  a point $\Pbf\in\Pi^n_{01}$ satisfies a shift equation $t_A(X,P)=s_A(X,P)$ iff $\Pbf$ satisfies any projection $\pi_b(t_A(X,P))=\pi_b(s_A(X,P))$ ($b\in B$). Below we use the following denotations: if $\Pbf$ satisfies the $b$-th projection of an equation $E_A(X,P)$ we write $\pi_b(E_A(\Pbf,P))$; otherwise, we write $\neg \pi_b(E_A(\Pbf,P))$. 

\bigskip

{\it Suppose below that a system $S(X)$ of semigroup equations and an equation $E(X)$ satisfy the inclusion~(\ref{eq:inclusion1}) over the semigroup $C$}.

Since the semigroup $B$ is equationally Noetherian (Theorem~\ref{th:N_is_noeth}), there exists a finite subsystem
\begin{equation}
\label{eq:S_hat}
\hat{S}(X)\subseteq S(X)\mbox{ with }Y_B=\V_B(S_B)=\V_B(\hat{S}_B).
\end{equation}

The set $Y_B$ generates the following equivalence relation over $\At_+$:
\[
t\sim s\Leftrightarrow t(P)=s(P)\mbox{ at each $P\in Y_B$.}
\]

\begin{lemma}
For a system $S(X)$ and an equation $E(X)\colon t(X)=s(X)$ it holds $t_B\sim s_B$.
\label{l:t_B_sim_s_B}
\end{lemma}
\begin{proof}
By the definition of a wreath product, we have $\V_B(S_B)\subseteq \V_B(t_B(X)=s_B(X))$ if~(\ref{eq:inclusion1}) holds. Hence, $t_B(P)=s_B(P)$ at any point $P\in\V_B(S_B(X))$, and we obtain $t_B\sim s_B$.
\end{proof}

Let us discuss the case, when the system $S_B(X)$ is inconsistent. It follows that $\hat{S}_B(X)$ is also inconsistent, and the inclusion~(\ref{eq:inclusion2}) obviously holds for the subsystem $\hat{S}$, since $\V_C(\hat{S})=\emptyset$. Thus, {\it we assume below that $S_B$ is consistent and the set $Y_B$ is not empty.}

Let $[t]$ denote the equivalence class of a term $t\in\At_+$ with respect to the equivalence relation $\sim$. The {\it length} $\|t\|\in \N$  of $t(X)=\al_1x_1+\ldots+\al_nx_n$ is $\al_1+\ldots+\al_n$.

\begin{lemma}
\label{l:[t]_is_finite}
The equivalence class $[t]$ is finite for any $t\in \At_+$. 
\end{lemma}
\begin{proof}
Assume there exist an infinitely many $\At_+$-terms $t_1,t_2,\ldots$ with $t_i\sim t$. Hence, the infinite system of additive equations $S_\infty=\{t(X)=t_i(X)\mid i\in \N\}$ is consistent, since $\V_B(S_\infty)\supseteq Y_B$.

Let $P\in \V_B(S_\infty)$ (i.e. $t(P)=t_i(P)$ for all $i\in \N$) and denote $b=t(P)$. Since the set $\{t_i\}$ is infinite, there exists a term $t_m$ such that $\|t_m\|>b$. By the definition of the length, we obtain $t_m(P)>b$ that contradicts the equality $b=t(P)=t_m(P)$. 
\end{proof}

%One can directly check that the set of all equivalence classes $\Gamma(Y_B)=\{[t]\mid t\in\At_+\}$ is a semigroup itself, and it is called {\it the coordinate semigroup} of the set $Y_B$.

Let us define a relation over $\At_+$ by:
\[
s<t \Leftrightarrow \exists s^\pr\in\At_+\colon s+s^\pr \sim t,
\]
and denote  $\downarrow t=\{s\mid s<t\}\subseteq\Tcal_+$.  

\begin{lemma} The relation $<$ has the following properties:
\begin{enumerate}
\item $<$ is a strict partial order;
\item if $s<t$ and $t\sim t^\pr$ then $s<t^\pr$;
\item if $t\sim t^\pr$ then $\downarrow t=\downarrow t^\pr$
\item if $s<t$ and $s\sim s^\pr$ then $s^\pr<t$.
\end{enumerate}
\label{l:properties_<}
\end{lemma}
\begin{proof}
\begin{enumerate}
\item The irreflexivity is obvious: terms $t$ and $t+s$ cannot be equivalent over $B$ for any $Y_B\neq \emptyset$. 
Let us prove the transitivity. If $t_1<t_2$ and $t_2<t_3$, there exist $\At_+$-terms $s_1,s_2$ such that $t_1+s_1\sim t_2$, $t_2+s_2\sim t_3$. Hence, $t_1+s_1+s_2\sim t_3$ and $t_1<t_3$.

\item There exists an $\At_+$-term  $s^\pr$ such that $s+s^\pr\sim t\sim t^\pr$. Hence, $s< t^\pr$.

\item This statement immediately follows from the previous one.

\item Let us take $s<t$ and $s\sim s^\pr$. It follows there exists an $\At_+$-term  $s_0$ such that $s+s_0\sim t$, i.e. $s(P)+s_0(P)=t(P)$ for each $P\in Y_B$. By the condition, $s(P)=s^\pr(P)$, hence $s^\pr(P)+s_0(P)=t(P)$ for each $P\in Y_B$, and therefore $s^\pr+s_0\sim t$. Thus, $s^\pr<t$.

\end{enumerate}  
\end{proof}

\begin{lemma}
The set $\downarrow t$ is finite for any $t\in \At_+$. 
\label{l:[t]_<_is_finite}
\end{lemma}
\begin{proof}
Assume terms $t_1,t_2,\ldots$ satisfy $t_i<t$. Hence, there exist terms $s_1,s_2,\ldots$ such that $t_i+s_i\sim t$. It follows $t_i(P)+s_i(P)=t(P)=b$ for each $P\in Y_B$. Since the set $\{t_i\}$ is infinite, there exists a term $t_m$ with $\|t_m\|>b$. By the definition of the length, we obtain $t_m(P)>b$ that contradicts the equality $t_m(P)+s_m(P)=t(P)$. 
\end{proof}

%Let $E$ be an $\LL$-equation. By $l(E),r(E)$ we denote respectively the left and right part of the equation. 

Let us take a system $S$ and an equation $E(X)\colon t(X)=s(X)$ which satisfy the inclusion~(\ref{eq:inclusion1}). By Lemma~\ref{l:t_B_sim_s_B}, we have $t_B\sim s_B$ and let 
\begin{equation}
T_<=\downarrow t_B\cup\{0\}=\downarrow s_B\cup\{0\},
\label{eq:T_<}
\end{equation}
be the set of terms with zero term $0$ adjointed. 

By Lemma~\ref{l:[t]_<_is_finite}, the set $T_<$ is finite. Let us define a system $S^\ast(X)\subseteq S(X)$ as follows:
\begin{equation}
\label{eq:S^ast}
S^\ast(X)=\hat{S}(X)\cup\{t^\pr(X)=s^\pr(X)\in S \mid \|t^\pr_B(X)\|\leq |T_<|\mbox{ or }
\|s^\pr_B(X)\|\leq |T_<|\},
\end{equation}
where the subsystem $\hat{S}(X)\subseteq S(X)$ was defined by~(\ref{eq:S_hat}).

\begin{lemma}
The system $S^\ast$ is finite.
\end{lemma}
\begin{proof}
Assume that $S^\ast$ is infinite. The set of all $\Tcal_+$-terms of the length at most $|T_<|$ is finite, hence, there exist a term $t^\pr(X)$, $\|t^\pr_B\|\leq |T_<|$ and an infinite subsystem $\tilde{S}=\{t^\pr(X)=s_i(X)\mid i\in I\}\subseteq S$. It follows that the equivalence class $[t^\pr_B]$ is infinite, and we obtain a contradiction with Lemma~\ref{l:[t]_is_finite}. 
\end{proof}

The following lemma explains the values of long terms.

\begin{lemma}
\label{l:long_terms}
Let $t(X)=x_{i_1}\ldots x_{i_k}$ be a term over a wreath product $C=A\wr B$, and a point $(\Pbf,P)$ is defined by~(\ref{eq:point_Pbf_and_P}). Suppose there exists a finite set of indexes  $B^\pr\subseteq B$ such that $p_{i,b}=0$ if $b\notin B^\pr$. Then $t_A(\Pbf,P)=(0,0,,\ldots)\in\Pi$ for the term $t_A(X,P)$ if $|B^\pr|<k$.
\end{lemma}
\begin{proof}
Assume that a projection $\pi_b(t_A(X,P))$~(\ref{eq:pi_b(t_A(X,P))}) is not equal to $0\in A$. It follows that 
\[
p_{i_1,b}\neq 0,
\; p_{i_2,b+p_{i_1}}\neq 0, \; \ldots \; p_{i_k,b+p_{i_1}+\ldots+ p_{i_{k-1}}}\neq 0.
\]
Hence, 
\begin{equation}
b,\; b+p_{i_1},\; \ldots \; b+p_{i_1}+\ldots+ p_{i_{k-1}}\in B^\pr.
\end{equation}
These elements are pairwise distinct in $B$, and we obtain the contradiction $|B^\pr|\geq k$. 
\end{proof}

\section{Equationally Noetherian wreath products}

In this section we find a wide class of wreath products $C=A\wr B$ which are not equationally Noetherian.

Recall that a semigroup $A$ with a zero $0$ is {\it nilpotent} if there exists a number $s\in\N$ such that any product $a_1a_2\ldots a_s$ ($a_i\in A$) equals $0$.

\noindent{\bf Theorem A.} {\it Suppose a semigroups $A$ contains zero and a semigroup $B$ is infinite cyclic. The wreath product $C=A\wr B$ is equationally Noetherian iff $A$ is nilpotent.
}
\begin{proof}
Let us prove the ``only if`` part of the statement. 

Consider an infinite system over $C$:
\[
S=\begin{cases}
x_1x_3=x_4 x_6,\\
x_1x_2x_3=x_4x_5 x_6,\\
x_1x_2^2x_3=x_4x_5^2 x_6,\\
x_1x_2^3x_3=x_4x_5^3 x_6,\\
\ldots
\end{cases}
\]
Let
\[
S_n=\begin{cases}
x_1x_3=x_4 x_6,\\
x_1x_2x_3=x_4x_5 x_6,\\
x_1x_2^2x_3=x_4x_5^2 x_6,\\
\ldots\\
x_1x_2^{n-1}x_3=x_4 x_5^{n-1} x_6.
\end{cases}
\]
be the first $n$ equations of $S$. 

Since $A$ is not nilpotent, there exist elements $a_1,a_2,\ldots,a_{n+1}\in A$ with $a_1a_2\ldots a_{n+1}\neq 0$.
Define a point $\Pbf=(\pbf_1,\ldots,\pbf_6)\in\Pi^6$ as follows:
\begin{eqnarray*}
\pbf_1=(a_1,0,0,\ldots),\\
\pbf_2=(a_1,a_2,\ldots,a_{n}),\\
\pbf_3=(\underbrace{0,\ldots,0}_{\mbox{$n+1$ times}},a_{n+1},0,0,\ldots),\\
\pbf_4=\pbf_5=\pbf_6=\zero,
\end{eqnarray*}
where $\zero=(0,0,\ldots)\in\Pi$. Denote $P=(1,1,1,1,1,1)\in B^6$.  
Let us prove that the point $(\Pbf,P)$ satisfies the system $S_n$. We have the following systems $(S_n)_A,(S_n)_B$:
\[
(S_n)_A=\begin{cases}
x_1\s_{x_1}(x_3)=x_4 \s_{x_4}(x_6),\\
x_1\s_{x_1}(x_2)\s_{x_1+x_2}(x_3)=x_4\s_{x_4}(x_5) \s_{x_4+x_5}(x_6),\\
x_1\s_{x_1}(x_2)\s_{x_1+x_2}(x_2)\s_{x_1+2x_2}(x_3)=x_4\s_{x_4}(x_5) \s_{x_4+x_5}(x_5)\s_{x_4+2x_5}(x_6),\\
\ldots\\
x_1\s_{x_1}(x_2)\s_{x_1+x_2}(x_2)\ldots\s_{x_1+(n-2)x_2}(x_2)\s_{x_1+(n-1)x_2}(x_3)=\\
x_4\s_{x_4}(x_5)\s_{x_4+x_5}(x_5)\ldots\s_{x_4+(n-2)x_5}(x_5)\s_{x_4+(n-1)x_5}(x_6)
\end{cases}
\]
\[
(S_n)_B=\begin{cases}
x_1+x_3=x_4+x_6,\\
x_1+x_2+x_3=x_4+x_5+ x_6,\\
x_1+2x_2+x_3=x_4+2x_5 +x_6,\\
\ldots\\
x_1+(n-1)x_2+x_3=x_4 +(n-1)x_5+ x_6.
\end{cases}
\]
Let us put the points $\Pbf,P$ into the system $(S_n)_A$. Notice that the right parts of all equations become $\zero$ (since $\pbf_4=\zero$):
\[
(S_n)_A(\Pbf,P)=\begin{cases}
\pbf_1\s_{1}(\pbf_3)=\zero,\\
\pbf_1\s_{1}(\pbf_2)\s_{2}(\pbf_3)=\zero,\\
\pbf_1\s_{1}(\pbf_2)\s_{2}(\pbf_2)\s_{3}(\pbf_3)=\zero,\\
\ldots\\
\pbf_1\s_{1}(\pbf_2)\s_{2}(\pbf_2)\ldots\s_{n-1}(\pbf_2)\s_{n}(\pbf_3)=\zero
\end{cases}
\Leftrightarrow
\]
\[
\Leftrightarrow\begin{cases}
(a_1,0,0,\ldots)(\underbrace{0,\ldots,0}_{\mbox{$n$ times}},a_{n+1},0,0,\ldots)=\zero,\\
(a_1,0,0,\ldots)(a_2,a_3,\ldots,a_{n},0,0,\ldots)(\underbrace{0,\ldots,0}_{\mbox{$n-1$ times}},a^\pr,0,0,\ldots)=\zero,\\
(a_1,0,0,\ldots)(a_2,a_3,\ldots,a_{n},0,0,\ldots)(a_3,a_4,\ldots,a_{n-1},a_n,0,\ldots)
(\underbrace{0,\ldots,0}_{\mbox{$n-2$ times}},a_{n+1},0,0,\ldots)=\zero,\\
\ldots\\
(a_1,0,0,\ldots)(a_2,a_3,\ldots,a_{n},0,0,\ldots)\ldots(a_{n-1},a_n,0,\ldots)(0,a_{n+1},0,0,\ldots)=\zero
\end{cases}
\]
and we see that all equations of $(S_n)_A$ are satisfied by $(\Pbf,P)$. 

The system $(S_n)_B$ is obviously satisfied by the point $P$. Thus, $(\Pbf,P)\in\V_C(S_n)$.

Let us take the $(n+1)$-th equation of $S_A$:
\begin{multline*}
x_1\s_{x_1}(x_2)\s_{x_1+x_2}(x_2)\ldots\s_{x_1+(n-1)x_2}(x_2)\s_{x_1+nx_2}(x_3)=\\
x_4\s_{x_4}(x_5)\s_{x_4+x_5}(x_5)\ldots\s_{x_4+(n-1)x_5}(x_5)\s_{x_4+nx_5}(x_6).
\end{multline*}
For this equation the point $(\Pbf,P)$ gives
\begin{eqnarray*}
\pbf_1\s_{1}(\pbf_2)\s_{2}(\pbf_2)\ldots\s_{n}(\pbf_2)\s_{n+1}(\pbf_3)&=&
(a_1,0,0,\ldots)\\
&&(a_2,a_3,\ldots,a_{n},0,0,\ldots)\\
&&(a_3,a_4,\ldots,a_{n},0,0,\ldots)\\
&&\ldots\\
&&(a_n,0,0,\ldots)\\
&&(a_{n+1},0,0,\ldots)\\
&=&(a_1a_2a_3\ldots a_{n+1},0,0,\ldots)\neq \zero.
\end{eqnarray*}

Thus, $(\Pbf,P)\notin\V_C(S)$, and the system $S_n$ is not equivalent to $S$ over $C$. Since we arbitrarily chose $n$, there does not exist a finite subsystem with the solution set $\V_C(S)$.

\medskip

Let us prove the ``if'' part of the statement. 

Suppose the semigroup $A$ is nilpotent, i.e. the product $a_1\ldots a_s$ equals $0$ for any $a_i$. Hence,  for a semigroup term $t(X)=x_{i_1}\ldots x_{i_k}$ with $k>s$ the wreath term $t_A(X)$ equals $(0,0,0,\ldots)$ at any point $(\Pbf,P)$, $\Pbf\in\Pi^n$, $P\in B^n$. We have that for a finite set of variables $X$ there exist at most finite number of pairwise non-equivalent wreath terms. Thus, any system $S(X)$ has a finite subsystem $S^\ast(X)\subseteq S(X)$ such that $S^\ast_A(X)$ and $S_A(X)$ are equivalent over $\Pi$. 

Since the semigroup $B$ is equationally Noetherian (Theorem~\ref{th:N_is_noeth}), there exists a finite subsystem $\hat{S}(X)\subseteq S(X)$ with $\V_B(S_B^\ast)=\V_B(S_B)$. 

Finally, the finite subsystem $S^\ast(X)\cup\hat{S}(X)$ is equivalent to $S(X)$ over $C$.
\end{proof}

\section{Main result}

Let us prove Theorem B. Let $A$ be a semigroup with zero and $B$ is infinite cyclic.

Below in this section a system $S(X)$ and an equation $E(X)$ satisfy the inclusion~(\ref{eq:inclusion1}).
Assume that the inclusion~(\ref{eq:inclusion2}) does not hold for the subsystem $S^\ast(X)$~(\ref{eq:S^ast}), i.e.
\[
\V_C(S^\ast)\nsubseteq \V_C(E),
\]
and there exists a point $(\Pbf,P)$~(\ref{eq:point_Pbf_and_P})
with $(\Pbf,P)\in\V_C(S^\ast)\setminus\V_C(E)$. Since $S^\ast(X)$ contains the subsystem $\hat{S}(X)$, then $P\in\V_B(E_B)$, but $\Pbf\notin \V_\Pi(E_A(X,P))$ and there exists a projection with $\neg \pi_\be(E_A(\Pbf,P))$. The following lemma states that we may put $\be=1$.

\begin{lemma}
For all denotations above there exists a point $\Rbf=(\rbf_1,\ldots,\rbf_n)$, $\rbf_i=(r_{i,1},r_{i,2},\ldots)\in \Pi$ such that
\begin{enumerate}
\item $\Rbf\in\V_\Pi(S^\ast_A(X,P))$;
\item $\neg \pi_1(E_A(\Rbf,P))$.
\end{enumerate}
\label{l:to_first_projection}
\end{lemma}
\begin{proof}
Let 
\[
\rbf_i=(p_{i,\be},p_{i,\be+1},\ldots)
\]
be the shift of the elements $\pbf_i\in\Pi$. Let us prove $\Rbf\in\V_\Pi(S^\ast_A(X,P))$. 

Take an equation $t(X)=s(X)\in S^\ast$ such that
the $b$-th projections of $t_A(X,P),s_A(X,P)$ are defined by~(\ref{eq:pi_b(t_A(X,P))},\ref{eq:pi_b(s_A(X,P))}).

We have
\begin{eqnarray*}
\pi_b(t_A(\Rbf,P))&=&r_{i_1,b}r_{i_2,b+p_{i_1}}\ldots r_{i_k,b+p_{i_1}+\ldots+ p_{i_{k-1}}}\\
&=&p_{i_1,b+(\be-1)}p_{i_2,b+p_{i_1}+(\be-1)}\ldots p_{i_k,b+p_{i_1}+\ldots+ p_{i_{k-1}}+(\be-1)}\\
&=& \pi_{b+(\be-1)}(t_A(\Pbf,P)),\\
\pi_b(s_A(\Rbf,P))&=&r_{j_1,b}r_{j_2,b+p_{j_1}}\ldots r_{j_l,b+p_{j_1}+\ldots+ p_{j_{l-1}}}\\
&=&p_{j_1,b+(\be-1)}p_{j_2,b+p_{j_1}+(\be-1)}\ldots p_{j_l,b+p_{j_1}+\ldots+ p_{j_{l-1}}+(\be-1)}\\
&=& \pi_{b+(\be-1)}(s_A(\Pbf,P)).
\end{eqnarray*}
Since $\Pbf\in\V_\Pi( S_A^\ast(X,P))$, it holds $\pi_{b+(\be-1)}(t_A(\Pbf,P))=\pi_{b+(\be-1)}(s_A(\Pbf,P))$, and we finally obtain $\pi_b(t_A(\Rbf,P))=\pi_b(s_A(\Rbf,P))$. Thus, $\Rbf\in\V_\Pi(S^\ast_A(X,P))$.

Similarly, for the equation $E_A(X,P)\colon t_A(X,P)=s_A(X,P)$ we have
\[
\pi_1(t_A(\Rbf,P))=\pi_{1+(\be-1)}(t_A(\Pbf,P))\neq \pi_{1+(\be-1)}(s_A(\Pbf,P))=\pi_1(s_A(\Rbf,P)),
\]
and therefore $\neg \pi_1(E_A(\Rbf,P))$.
\end{proof}

{According to Lemma~\ref{l:to_first_projection}, we assume below that the point $\Pbf$ does not satisfy the first ($\beta=1$) projection of the shift equation $E_A(X,P)$, i.e.  }
\begin{equation*}
\neg \pi_1(E_A(\Pbf,P)).
\label{eq:Pbf_not_1_projection}
\end{equation*}

\bigskip

Let us introduce new variables $X^\pr=\{x_i^\pr\mid 1\leq i\leq n\}$, $X^{\pr\pr}=\{x_i^{\pr\pr}\mid 1\leq i\leq n\}$ and obtain a new system $S^\ast(X^\pr,X^{\pr\pr})$ by
\begin{equation}
x_i=x_i^\pr x_i^{\pr\pr}.
\label{eq:new_variables}
\end{equation}

\begin{remark}
Let us explain the sense of new variables. We know that the last lower indexes $x_{i_1}+\ldots+ x_{i_{k-1}}$, $x_{j_1}+\ldots+ x_{j_{l-1}}$ in terms $t_A(X),s_A(X)$~(\ref{eq:t_A(X)},\ref{eq:s_A(X)}) are not equal to the parts of the additive equation $t_B(X)=s_B(X)$. New variables allow us to achieve this property, and we essentially use it in Lemma~\ref{l:big_lemma}.
\end{remark}

\medskip

Clearly, the wreath terms $t_A(X),s_A(X)$~(\ref{eq:t_A(X)},\ref{eq:s_A(X)}) become
\begin{multline}
\label{eq:t_A(pr,prpr)}
t_A(X^\pr,X^{\pr\pr})=x^\pr_{i_1}\s_{x^\pr_{i_1}}(x^{\pr\pr}_{i_1})
\s_{x^\pr_{i_1}+x^{\pr\pr}_{i_1}}(x^\pr_{i_2})
\s_{x^\pr_{i_1}+x^{\pr\pr}_{i_1}+x^\pr_{i_2}}(x^{\pr\pr}_{i_2})\\
\ldots
\s_{x^\pr_{i_1}+\ldots+ x^{\pr\pr}_{i_{k-1}}}(x^\pr_{i_k})\s_{x^\pr_{i_1}+\ldots+ x^{\pr\pr}_{i_{k-1}}+x^\pr_{i_k}}(x^{\pr\pr}_{i_k}),
\end{multline}

\begin{multline}
s_A(X^\pr,X^{\pr\pr})=x^\pr_{j_1}\s_{x^\pr_{j_1}}(x^{\pr\pr}_{j_1})
\s_{x^\pr_{j_1}+x^{\pr\pr}_{j_1}}(x^\pr_{j_2})
\s_{x^\pr_{j_1}+x^{\pr\pr}_{j_1}+x^\pr_{j_2}}(x^{\pr\pr}_{j_2})\\
\ldots
\s_{x^\pr_{j_1}+\ldots+ x^{\pr\pr}_{j_{l-1}}}(x^\pr_{j_l})\s_{x^\pr_{j_1}+\ldots+ x^{\pr\pr}_{j_{l-1}}+x^\pr_{j_l}}(x^{\pr\pr}_{j_l}).
\label{eq:s_A(pr,prpr)}
\end{multline}
The shift terms for some points $P^\pr=(p_1^\pr,\ldots,p_n^\pr)\in B^n$, $P^{\pr\pr}=(p_1^{\pr\pr},\ldots,p_n^{\pr\pr})\in B^n$ are the following
\begin{multline}
\label{eq:t_A(pr,prpr,P)}
t_A(X^\pr,X^{\pr\pr},P^\pr,P^{\pr\pr})=x^\pr_{i_1}\s_{p^\pr_{i_1}}(x^{\pr\pr}_{i_1})
\s_{p^\pr_{i_1}+p^{\pr\pr}_{i_1}}(x^\pr_{i_2})
\s_{p^\pr_{i_1}+p^{\pr\pr}_{i_1}+p^\pr_{i_2}}(x^{\pr\pr}_{i_2})\\
\ldots 
\s_{p^\pr_{i_1}+\ldots+ p^{\pr\pr}_{i_{k-1}}}(x^\pr_{i_k})
\s_{p^\pr_{i_1}+\ldots+p^{\pr\pr}_{i_{k-1}}+p^\pr_{i_k}}(x^{\pr\pr}_{i_k}),
\end{multline}
\begin{multline}
s_A(X^\pr,X^{\pr\pr},P^\pr,P^{\pr\pr})=x^\pr_{j_1}\s_{p^\pr_{j_1}}(x^{\pr\pr}_{j_1})
\s_{p^\pr_{j_1}+p^{\pr\pr}_{j_1}}(x^\pr_{j_2})
\s_{p^\pr_{j_1}+p^{\pr\pr}_{j_1}+p^\pr_{j_2}}(x^{\pr\pr}_{j_2})\\
\ldots 
\s_{p^\pr_{j_1}+\ldots+ p^{\pr\pr}_{j_{l-1}}}(x^\pr_{j_l})
\s_{p^\pr_{j_1}+\ldots+ p^{\pr\pr}_{j_{l-1}}+p^\pr_{j_l}}(x^{\pr\pr}_{j_l}).
\label{eq:s_A(pr,prpr,P)}
\end{multline}

For the shift terms $t_A(X^\pr,X^{\pr\pr},P^\pr,P^{\pr\pr})$, $s_A(X^\pr,X^{\pr\pr},P^\pr,P^{\pr\pr})$ one can obviously get the $b$-projections:
\begin{multline}
\label{eq:pi_b(t_A(pr,prpr,P))}
\pi_b(t_A(X^\pr,X^{\pr\pr},P^\pr,P^{\pr\pr}))=
x^\pr_{i_1,b}
x^{\pr\pr}_{i_1,b+p^\pr_{i_1}}
x^\pr_{i_2,b+p^\pr_{i_1}+p^{\pr\pr}_{i_1}}
x^{\pr\pr}_{i_2,b+p^\pr_{i_1}+p^{\pr\pr}_{i_1}+p^\pr_{i_2}}\\
\ldots 
x^\pr_{i_k,b+p^\pr_{i_1}+\ldots+ p^{\pr\pr}_{i_{k-1}}}
x^{\pr\pr}_{i_k,b+p^\pr_{i_1}+\ldots+ p^{\pr\pr}_{i_{k-1}}+p^\pr_{i_k}},
\end{multline}
\begin{multline}
\pi_b(s_A(X^\pr,X^{\pr\pr},P^\pr,P^{\pr\pr}))=
x^\pr_{j_1,b}
x^{\pr\pr}_{j_1,b+p^\pr_{j_1}}
x^\pr_{j_2,b+{p^\pr_{j_1}+p^{\pr\pr}_{j_1}}}
x^{\pr\pr}_{j_2,b+p^\pr_{j_1}+p^{\pr\pr}_{j_1}+p^\pr_{j_2}}\\
\ldots 
x^\pr_{j_l,b+p^\pr_{j_1}+\ldots+ p^{\pr\pr}_{j_{l-1}}}
x^{\pr\pr}_{j_l,b+p^\pr_{j_1}+\ldots+ p^{\pr\pr}_{j_{l-1}}+p^\pr_{j_l}}.
\label{eq:pi_b(s_A(pr,prpr,P))}
\end{multline}

The system $S^\ast(X^\pr,X^{\pr\pr})$ in variables $(X^\pr,X^{\pr\pr})$ is considered below as a system over the semigroup $C_{01}$. 

Let us take arbitrary points  
\[
\Pbf^\pr=(\pbf^\pr_1,\ldots,\pbf^\pr_n),\; \pbf^\pr_i=(p^\pr_{i,0},p^\pr_{i,1},\ldots)\in\Pi_{01},
\] 
\[
\Pbf^{\pr\pr}=(\pbf^{\pr\pr}_1,\ldots,\pbf^{\pr\pr}_n),\; \pbf^{\pr\pr}_i=(p^{\pr\pr}_{i,0},p^{\pr\pr}_{i,1},\ldots)\in\Pi_{01}.
\] 
and assign
\begin{equation*}
\pbf^\pr_i:\approx \pbf_i,\; \pbf^{\pr\pr}_i:=\one=(1,1,1,\ldots)
\end{equation*}
In other words, 
\begin{equation*}
p^\pr_{i,b}=p_{i,(b-1)},\; p^{\pr\pr}_{i,b}=1
\end{equation*}
where the elements $p_{i,b}$ were defined by~(\ref{eq:point_Pbf_and_P}) and $1$ is the unit of the semigroup $A_1$. 

\begin{lemma}
\label{l:double_variables}
Suppose a point $(\Pbf,P)$~(\ref{eq:point_Pbf_and_P}) is a solution of a system $S^\ast$ over $C$. Let us consider the system $S^\ast(X^\pr,X^{\pr\pr})$ and equation $E(X^\pr,X^{\pr\pr})$ in variables~(\ref{eq:new_variables}) over the semigroup $C_{01}$. We have:
\begin{enumerate}
\item $(P,\bar{0})\in\V_{B_0}(S^\ast_B(X^\pr,X^{\pr\pr}))$, where 
\[
\bar{0}=(0,\ldots,0)\in B^n;
\]
\item $(\Pbf^\pr,\Pbf^{\pr\pr})\in\V_{\Pi_{01}}(S^\ast_A(X^\pr,X^{\pr\pr},P,\bar{0}))$;
\item $\neg \pi_0(E_A(\Pbf^\pr,\Pbf^{\pr\pr},P,\bar{0}))$.
\end{enumerate}
\end{lemma}
\begin{proof}
\begin{enumerate}
\item In variables $X^\pr,X^{\pr\pr}$ any $\At_+$-equation $s_1(X)=s_2(X)\in S^\ast_B(X)$
\[
s_1(X)=\al_1x_{i_1}+\ldots+\al_kx_{i_k},\; s_2(X)=\be_1x_{j_1}+\ldots+\be_lx_{j_l}
\]
becomes
\[
\al_1x_{i_1}^\pr+\ldots+\al_kx_{i_k}^\pr+\al_1x_{i_1}^{\pr\pr}+\ldots+\al_kx_{i_k}^{\pr\pr}=
\be_1x_{j_1}^\pr+\ldots+\be_lx_{j_l}^\pr+\be_1x_{j_1}^{\pr\pr}+\ldots+\be_lx_{j_l}^{\pr\pr}\in S_B(X^\pr,X^{\pr\pr}).
\]
The point $(P,\bar{0})$ obviously satisfies this equation, since $P\in\V_B(s_1(X)=s_2(X))$.

\item Let us take an equation $t_A(X^\pr,X^{\pr\pr})=s_A(X^\pr,X^{\pr\pr})\in S^\ast_A(X^\pr,X^{\pr\pr})$ defined by~(\ref{eq:t_A(pr,prpr)},\ref{eq:s_A(pr,prpr)}). We have

\begin{multline}
\label{eq:t_A(Pbf,one,P,zero)}
t_A(\Pbf^\pr,\Pbf^{\pr\pr},P,\bar{0})=\pbf_{i_1}\s_{p_{i_1}}(\one)
\s_{p_{i_1}}(\pbf_{i_2})
\s_{p_{i_1}+p_{i_2}}(\one)
\ldots 
\s_{p_{i_1}+\ldots+ p_{i_{k-1}}}(\pbf_{i_k})
\s_{p_{i_1}+\ldots+p_{i_k}}(\one)=\\
\pbf_{i_1}
\s_{p_{i_1}}(\pbf_{i_2})
\ldots 
\s_{p_{i_1}+\ldots+ p_{i_{k-1}}}(\pbf_{i_k})\approx t_A(\Pbf,P).
\end{multline}
\begin{multline}
\label{eq:s_A(Pbf,one,P,zero)}
s_A(\Pbf^\pr,\Pbf^{\pr\pr},P,\bar{0})=\pbf_{j_1}\s_{p_{j_1}}(\one)
\s_{p_{j_1}}(\pbf_{j_2})
\s_{p_{j_1}+p_{j_2}}(\one)
\ldots 
\s_{p_{j_1}+\ldots+ p_{j_{l-1}}}(\pbf_{j_l})
\s_{p_{j_1}+\ldots+p_{j_l}}(\one)=\\
\pbf_{j_1}
\s_{p_{j_1}}(\pbf_{j_2})
\ldots 
\s_{p_{j_1}+\ldots+ p_{j_{l-1}}}(\pbf_{j_l})\approx s_A(\Pbf,P).
\end{multline}
Since $(\Pbf,P)$ is a solution of $S^\ast(X)$, we have $t_A(\Pbf,P)=s_A(\Pbf,P)$ and therefore $t_A(\Pbf^\pr,\Pbf^{\pr\pr},P,\bar{0})=s_A(\Pbf^\pr,\Pbf^{\pr\pr},P,\bar{0})$.

\item Let $E(X)\colon t(X)=s(X)$ and assume that the wreath terms $t_A(X^\pr,X^{\pr\pr}),s_A(X^\pr,X^{\pr\pr})$ are defined by~(\ref{eq:t_A(pr,prpr)},\ref{eq:s_A(pr,prpr)}). By~(\ref{eq:t_A(Pbf,one,P,zero)},\ref{eq:s_A(Pbf,one,P,zero)}), we obtain 
\begin{eqnarray*}
t_A(\Pbf^\pr,\Pbf^{\pr\pr},P,\bar{0})\approx t_A(\Pbf,P),\\
s_A(\Pbf^\pr,\Pbf^{\pr\pr},P,\bar{0})\approx s_A(\Pbf,P),
\end{eqnarray*}
and
\begin{eqnarray*}
\pi_0(t_A(\Pbf^\pr,\Pbf^{\pr\pr},P,\bar{0}))=\pi_1(t_A(\Pbf,P)),\\
\pi_0(s_A(\Pbf^\pr,\Pbf^{\pr\pr},P,\bar{0}))=\pi_1(s_A(\Pbf,P)).
\end{eqnarray*}
By Lemma~\ref{l:to_first_projection},  $\pi_1(t_A(\Pbf,P))\neq \pi_1(s_A(\Pbf,P))$. Thus, $\pi_0(t_A(\Pbf^\pr,\Pbf^{\pr\pr},P,\bar{0}))\neq \pi_0(s_A(\Pbf^\pr,\Pbf^{\pr\pr},P,\bar{0}))$. 
\end{enumerate}
\end{proof}

%Denote 
%\begin{equation}
%\Pbf^\pr=\Pbf,\; \Pbf^{\pr\pr}=\bar{\one},
%\end{equation}
%where the point $\Pbf$ is defined by~(\ref{eq:point_Pbf_and_P}) and $\bar{\one}$ was defined in Lemma~\ref{l:double_variables}. The coordinates of $\Pbf^\pr,\Pbf^{\pr\pr}$ are denoted by 
%\[
%\Pbf^\pr=(\pbf^\pr_1,\ldots,\pbf^\pr_n),\; \pbf^\pr_i=(p^\pr_{i,1},p^\pr_{i,2},\ldots)\in\Pi,
%\] 
%\[
%\Pbf^{\pr\pr}=(\pbf^{\pr\pr}_1,\ldots,\pbf^{\pr\pr}_n),\; \pbf^{\pr\pr}_i=(p^{\pr\pr}_{i,1},p^{\pr\pr}_{i,2},\ldots)\in\Pi.
%\] 

%Let us take a wreath equation $t_A(X)=s_A(X)\in S^\ast_A(X)$. The {\it index set} of this equation is the set of all $\At_+$-terms occurring in lower indexes of $\s$-operators. For example, if $t_A(X),s_A(X)$ are given by~(\ref{eq:t_A(X)},\ref{eq:s_A(X)}) then the index set of the equation $t_A(X)=s_A(X)$ is
%\[
%\{{x_{i_1}},{x_{i_1}}+{x_{i_2}},\ldots {x_{i_1}+\ldots+ x_{i_{k-1}}},{x_{j_1}},{x_{j_1}}+{x_{j_2}},\ldots {x_{j_1}+\ldots+ x_{j_{l-1}}}\}.
%\]

Let us define the following set of $\Tcal_+$-terms. 
\[
T_\ast=\{\downarrow t_B\mid t_B(X)=s_B(X)\in S^\ast_B(X)\}.
\]
Since $\downarrow t_B=\downarrow s_B$ for any equation $t_B(X)=s_B(X)\in S^\ast$, the set $T_\ast$ contains all sub-terms of the terms from the system $S^\ast_B(X)$. By the finiteness of $S^\ast(X)$ and Lemma~\ref{l:properties_<}, the set $T_\ast$ is finite.

 Put
\begin{equation}
\label{eq:T}
T=\{t^\pr+s^\pr\mid t\in T_<, s\in T_\ast\}\subseteq \Tcal_+,
\end{equation}
where the set $T_<$ was defined by~(\ref{eq:T_<}). Since the sets $T_<,T_\ast$ are finite, so is $T$.

The following lemma immediately follows from Theorem~\ref{th:N_discrimination_new}.

\begin{lemma}
For the set of terms $T$~(\ref{eq:T}) and the system $S^\ast(X)$~(\ref{eq:S^ast}) there exists a point $Q\in\V_B(S^\ast_B)$ such that 
\begin{equation}
t_i(Q)\neq t_j(Q) \mbox{ for any $t_i,t_j\in T$ if $t_i\nsim t_j$}.
\label{eq:Q}
\end{equation}
\label{l:point_for_discrimination}
\end{lemma}

Let us define new points $\Qbf^\pr,\Qbf^{\pr\pr}$ by 
\begin{eqnarray}
\label{eq:Qbf'}
\Qbf^\pr=(\qbf_1^\pr,\ldots,\qbf_n^\pr),\; \qbf_i^\pr=(q^\pr_{i,0},q^\pr_{i,1},\ldots)\in\Pi_{01},\\
\label{eq:Qbf''}
\Qbf^{\pr\pr}=(\qbf_1^{\pr\pr},\ldots,\qbf_n^{\pr\pr}),\; \qbf_i^{\pr\pr}=(q^{\pr\pr}_{i,0},q^{\pr\pr}_{i,1},\ldots)\in\Pi_{01}
\end{eqnarray}
The coordinates of the points $\Qbf^\pr,\Qbf^{\pr\pr}$ are defined by the points $\Pbf^\pr,\Pbf^{\pr\pr}$ as follows:
\begin{eqnarray*}
\label{eq:q'}
q^\pr_{i,b}=
\begin{cases}p^\pr_{i,c} \mbox{ if }\exists s\in T_<\mbox{ such that } b=s(Q), c=s(P),\\
0,\mbox{ otherwise}
\end{cases}\\
q^{\pr\pr}_{i,b}=\begin{cases}p^{\pr\pr}_{i,c} \mbox{ if }\exists s\in T_<\mbox{ such that } b=s(Q), c=s(P),\\
0,\mbox{ otherwise}
\end{cases}
\label{eq:q''}
\end{eqnarray*}
In particular, $q^\pr_{i,0}=p^\pr_{i,0}$ and $q^{\pr\pr}_{i,0}=p^{\pr\pr}_{i,0}$, since $0\in T_<$.

Let us show that the definition of the points $\Qbf^\pr,\Qbf^{\pr\pr}$ is well-defined. If assume the existence of two terms $s_1,s_2\in T_<$ with $b=s_1(Q)=s_2(Q)$, then the choice of $Q$ gives $s_1\sim s_2$ and therefore $s_1(P)=s_2(P)$.

\begin{lemma}
We have $(\Qbf^\pr,\Qbf^{\pr\pr})\in \V_{\Pi_{01}}(S^\ast_A(X^\pr,X^{\pr\pr},Q,\bar{0}))$, $(Q,\bar{0})\in\V_{B_0}(S^\ast_B(X^\pr,X^{\pr\pr}))$.
\label{l:big_lemma}
\end{lemma}
\begin{proof}
Since the point $Q$ is a solution of the system $S^\ast_B(X)$, the point $(Q,\bar{0})$ satisfies the system $S^\ast_B(X^\pr,X^{\pr\pr})$ (Lemma~\ref{l:double_variables}). Let us prove that the point $(\Qbf^\pr,\Qbf^{\pr\pr})$ satisfies $S^\ast_A(X^\pr,X^{\pr\pr},Q,\bar{0})$.

Let us take an arbitrary wreath equation $t(X^\pr,X^{\pr\pr})=s(X^\pr,X^{\pr\pr})\in S^\ast(X^\pr,X^{\pr\pr})$, where the wreath terms $t_A(X^\pr,X^{\pr\pr})$, $s_A(X^\pr,X^{\pr\pr})$ are defined by~(\ref{eq:t_A(pr,prpr)},\ref{eq:s_A(pr,prpr)}). This equation becomes the shift equation  $t_A(X^\pr,X^{\pr\pr},Q,\bar{0})=s_A(X^\pr,X^{\pr\pr},Q,\bar{0})$, where 
\begin{eqnarray}
\label{eq:t_A(pr,prpr,Q)}
t_A(X^\pr,X^{\pr\pr},Q,\bar{0})=
x^\pr_{i_1}\s_{q_{i_1}}(x^{\pr\pr}_{i_1})
\s_{q_{i_1}}(x^\pr_{i_2})
\s_{q_{i_1}+q_{i_2}}(x^{\pr\pr}_{i_2})
\ldots 
\s_{q_{i_1}+\ldots+ q_{i_{k-1}}}(x^\pr_{i_k})
\s_{q_{i_1}+\ldots+q_{i_k}}(x^{\pr\pr}_{i_k}),\\
s_A(X^\pr,X^{\pr\pr},Q,\bar{0})=
x^\pr_{j_1}\s_{q_{j_1}}(x^{\pr\pr}_{j_1})
\s_{q_{j_1}}(x^\pr_{j_2})
\s_{q_{j_1}+q_{i_2}}(x^{\pr\pr}_{j_2})
\ldots 
\s_{q_{j_1}+\ldots+ q_{j_{l-1}}}(x^\pr_{j_l})
\s_{q_{j_1}+\ldots+q_{j_l}}(x^{\pr\pr}_{j_l}),
\label{eq:s_A(pr,prpr,Q)}
\end{eqnarray}
and
\begin{eqnarray*}
\label{eq:pi_t_A(pr,prpr,Q)}
\pi_b(t_A(\Qbf^\pr,\Qbf^{\pr\pr},Q,\bar{0}))=
q^\pr_{i_1,b}
q^{\pr\pr}_{i_1,b+q_{i_1}}
q^\pr_{i_2,b+q_{i_1}}
q^{\pr\pr}_{i_2,b+q_{i_1}+q_{i_2}}
\ldots 
q^\pr_{i_k,b+q_{i_1}+\ldots+ q_{i_{k-1}}}
q^{\pr\pr}_{i_k,b+q_{i_1}+\ldots+q_{i_k}},\\
\pi_b(s_A(\Qbf^\pr,\Qbf^{\pr\pr},Q,\bar{0}))=
q^\pr_{j_1,b}
q^{\pr\pr}_{j_1,b+q_{j_1}}
q^\pr_{j_2,b+q_{j_1}}
q^{\pr\pr}_{j_2,b+q_{j_1}+q_{i_2}}
\ldots 
q^\pr_{j_l,b+q_{j_1}+\ldots+ q_{j_{l-1}}}
q^{\pr\pr}_{j_l,b+q_{j_1}+\ldots+q_{j_l}}.
\label{eq:pi_s_A(pr,prpr,Q)}
\end{eqnarray*}

If both expressions $\pi_b(t_A(\Qbf^\pr,\Qbf^{\pr\pr},Q,\bar{0}))$, $\pi_b(s_A(\Qbf^\pr,\Qbf^{\pr\pr},Q,\bar{0}))$ equal zero, the equality $\pi_b(t_A(\Qbf^\pr,\Qbf^{\pr\pr},Q,\bar{0}))=\pi_b(s_A(\Qbf^\pr,\Qbf^{\pr\pr},Q,\bar{0}))$ holds. 

Otherwise, we assume $\pi_b(t_A(\Qbf^\pr,\Qbf^{\pr\pr},Q,\bar{0}))\neq 0$, and it follows 
\begin{eqnarray*}
q^\pr_{i_1,b}\neq 0,\\
q^{\pr\pr}_{i_1,b+q_{i_1}}\neq 0,\\
q^\pr_{i_2,b+q_{i_1}}\neq 0,\\
q^{\pr\pr}_{i_2,b+q_{i_1}+q_{i_2}}\neq 0,\\
\ldots \\
q^\pr_{i_k,b+q_{i_1}+\ldots+ q_{i_{k-1}}}\neq 0,\\
q^{\pr\pr}_{i_k,b+q_{i_1}+\ldots+q_{i_k}}\neq 0.\\
\end{eqnarray*}  

By the definition of the points $\Qbf^\pr,\Qbf^{\pr\pr}$, there exist $\At_+$-terms $t_0,t_1,\ldots,t_k\in T_<$ such that
\begin{eqnarray*}
b=t_0(Q),\\
b+q_{i_1}=t_1(Q),\\
b+q_{i_1}+q_{i_2}=t_2(Q),\\
\ldots\\
b+q_{i_1}+\ldots+q_{i_k}=t_k(Q).
\end{eqnarray*}
Also we have the following term equalities at the point $Q$:
\begin{eqnarray*}
\left(t_0+x_{i_1}\right)(Q)=t_1(Q),\\
\left(t_0+x_{i_1}+x_{i_2}\right)(Q)=t_2(Q),\\
\ldots\\
\left(t_0+x_{i_1}+\ldots+x_{i_k}\right)(Q)=t_k(Q).
\end{eqnarray*}
Since the equation $t(X^\pr,X^{\pr\pr})=s(X^\pr,X^{\pr\pr})$ defines the additive equation $t_B(X)=s_B(X)\in S_B(X)$~(\ref{eq:t_B(X)},\ref{eq:s_B(X)}), then
\[
x_{i_1},\; x_{i_1}+x_{i_2},\; \ldots,\; x_{i_1}+\ldots+x_{i_k}\; \in T_\ast,
\]
and moreover
\[
t_0+x_{i_1},\; t_0+x_{i_1}+x_{i_2},\;\ldots\; t_0+x_{i_1}+\ldots+x_{i_k}\; \in T.
\]
The choice of the point $Q$ gives 
\begin{eqnarray*}
t_0(X)+x_{i_1}\sim t_1(X),\\
t_0(X)+x_{i_1}+x_{i_2}\sim t_2(X),\\
\ldots\\
t_0(X)+x_{i_1}+\ldots+x_{i_k}\sim t_k(X).
\end{eqnarray*}
Thus, we have
\begin{eqnarray*}
q^\pr_{i_1,b}=q^\pr_{i_1,t_0(Q)}=p^\pr_{i_1,t_0(P)},\\
q^{\pr\pr}_{i_1,b+q_{i_1}}=q^{\pr\pr}_{i_1,t_1(Q)}=p^{\pr\pr}_{i_1,t_1(P)}=p^{\pr\pr}_{i_1,t_0(P)+p_{i_1}},\\
q^\pr_{i_2,b+q_{i_1}}=q^{\pr}_{i_2,t_1(Q)}=p^{\pr}_{i_2,t_1(P)}=p^{\pr}_{i_2,t_0(P)+p_{i_1}},\\
q^{\pr\pr}_{i_2,b+q_{i_1}+q_{i_2}}=q^{\pr\pr}_{i_1,t_2(Q)}=p^{\pr\pr}_{i_1,t_2(P)}=p^{\pr\pr}_{i_1,t_0(P)+p_{i_1}+p_{i_2}},\\
\ldots\\
q^\pr_{i_l,b+q_{i_1}+\ldots+ q_{i_{k-1}}}=q^\pr_{i_l,t_{k-1}(Q)}=p^\pr_{i_l,t_{k-1}(P)}=p^{\pr}_{i_2,t_0(P)+p_{i_1}+\ldots p_{i_{k-1}}},\\
q^{\pr\pr}_{i_l,b+q_{i_1}+\ldots+ q_{i_{k-1}}+q_{i_k}}=q^{\pr\pr}_{i_l,t_k(Q)}=p^{\pr\pr}_{i_l,t_k(P)}=p^{\pr\pr}_{i_1,t_0(P)+p_{i_1}+\ldots+p_{i_k}}.
\end{eqnarray*}

By Lemma~\ref{l:t_B_sim_s_B}, the point $Q\in\V_B(S^\ast_B)$ satisfies the additive equation $t_B(X)=s_B(X)$, hence we have the equality  $q_{i_1}+\ldots+q_{i_k}=q_{j_1}+\ldots+q_{j_l}$. Hence,  $t_k(Q)=b+q_{j_1}+\ldots+q_{j_l}$, and 
\[
\left(t_0+x_{j_1}+\ldots+ x_{j_l}\right)(Q)=t_k(Q).
\]
Since $t_0\in T_<$ and $x_{j_1}+\ldots+ x_{j_l}\in T_{\ast}$, the $\At_+$-term $t_0(X)+x_{j_1}+\ldots+x_{j_l}$ belongs to $T$. By the choice of the point $Q$, it follows that 
\[
t_0(X)+x_{j_1}+\ldots+x_{j_l}\sim t_k(X). 
\] 
Since $t_k\in T_<$, the following terms also belong to the set $T_<$:
\begin{eqnarray*}
s_l(X)\sim t_0(X)+x_{j_1}+\ldots+x_{j_l},\\
s_{l-1}(X)\sim t_1(X)+x_{j_1}+\ldots+x_{j_{l-1}},\\
\ldots\\
s_1(X)\sim t_0(X)+x_{j_1},\\
s_0(X)\sim t_0(X).
\end{eqnarray*}
We have
\begin{eqnarray*}
q^\pr_{j_1,b}=q^\pr_{j_1,s_0(Q)}=p^\pr_{j_1,s_0(P)},\\
q^{\pr\pr}_{j_1,b+q_{j_1}}=q^{\pr\pr}_{j_1,s_1(Q)}=p^{\pr\pr}_{j_1,s_1(P)}=p^{\pr\pr}_{j_1,s_0(P)+p_{j_1}},\\
q^\pr_{j_2,b+q_{j_1}}=q^{\pr}_{j_2,s_1(Q)}=p^{\pr}_{j_2,s_1(P)}=p^\pr_{j_1,s_0(P)+p_{j_1}},\\
q^{\pr\pr}_{j_2,b+q_{j_1}+q_{i_2}}=q^{\pr\pr}_{j_1,s_2(Q)}=p^{\pr\pr}_{j_1,s_2(P)}=p^{\pr\pr}_{j_1,s_0(P)+p_{j_1}+p_{j_2}},\\
\ldots\\
q^\pr_{j_l,b+q_{j_1}+\ldots+ q_{j_{l-1}}}=q^\pr_{j_l,s_{l-1}(Q)}=p^\pr_{j_l,s_{l-1}(P)}=p^\pr_{j_1,s_0(P)+p_{j_1}+\ldots+p_{l-1}},\\
q^{\pr\pr}_{j_l,b+q_{j_1}+\ldots+ q_{j_{l-1}}+q_{j_l}}=q^{\pr\pr}_{j_l,s_l(Q)}=p^{\pr\pr}_{j_l,s_l(P)}=p^{\pr\pr}_{j_1,s_0(P)+p_{j_1}+\ldots+p_{j_l}}.
\end{eqnarray*}

Finally, we obtain (below $a=t_0(P)=s_0(P)$)
\begin{eqnarray*}
\pi_b(t_A(\Qbf^\pr,\Qbf^{\pr\pr},Q,\bar{0}))=
p^\pr_{i_1,a}
p^{\pr\pr}_{i_1,a+p_{i_1}}
p^{\pr}_{i_2,a+p_{i_1}}
p^{\pr\pr}_{i_1,a+p_{i_1}+p_{i_2}}
\ldots
p^{\pr}_{i_2,a+p_{i_1}+\ldots p_{i_{k-1}}}
p^{\pr\pr}_{i_1,a+p_{i_1}+\ldots+p_{i_k}}\\
=\pi_a(t_A(\Pbf^\pr,\Pbf^{\pr\pr},P,\bar{0})),\\
\pi_b(s_A(\Qbf^\pr,\Qbf^{\pr\pr},Q,\bar{0}))=
p^\pr_{j_1,a}
p^{\pr\pr}_{j_1,a+p_{j_1}}
p^\pr_{j_1,a+p_{j_1}}
p^{\pr\pr}_{j_1,a+p_{j_1}+p_{j_2}}
p^\pr_{j_1,a+p_{j_1}+\ldots+p_{l-1}}
p^{\pr\pr}_{j_1,a+p_{j_1}+\ldots+p_{j_l}}\\
=\pi_a(s_A(\Pbf^\pr,\Pbf^{\pr\pr},P,\bar{0})).
\end{eqnarray*}

Since $(\Pbf^\pr,\Pbf^{\pr\pr})\in\V_{\Pi_{01}}(S^\ast(X^\pr,X^{\pr\pr},P,\bar{0}))$, it follows $\pi_b(t_A(\Qbf^\pr,\Qbf^{\pr\pr},Q,\bar{0}))=\pi_b(s_A(\Qbf^\pr,\Qbf^{\pr\pr},Q,\bar{0}))$. Thus, $(\Qbf^\pr,\Qbf^{\pr\pr})\in\V_{\Pi_{01}}(S^\ast(X^\pr,X^{\pr\pr},Q,\bar{0}))$.
\end{proof}

\begin{lemma}
For the point $(\Qbf^\pr,\Qbf^{\pr\pr})$~(\ref{eq:Qbf'},\ref{eq:Qbf''}) and the equation $E(X)\colon t(X)=s(X)$ we have  $\neg \pi_0(E_A(\Qbf^\pr,\Qbf^{\pr\pr},Q,\bar{0}))$.
\end{lemma}
\begin{proof}
Let $E_A(X^\pr,X^{\pr\pr},Q,\bar{0})\colon t_A(X^\pr,X^{\pr\pr},Q,\bar{0})=s_A(X^\pr,X^{\pr\pr},Q,\bar{0})$, where the terms $t_A,s_A$ are defined by~(\ref{eq:t_A(pr,prpr,Q)},\ref{eq:s_A(pr,prpr,Q)}). We consider the $0$-th projection of $E_A(X^\pr,X^{\pr\pr},Q,\bar{0})$:
\begin{eqnarray*}
%\label{eq:pi_0_t_A(pr,prpr,Q)}
\pi_0(t_A(\Qbf^\pr,\Qbf^{\pr\pr},Q,\bar{0}))=
q^\pr_{i_1,0}
q^{\pr\pr}_{i_1,q^\pr_{i_1}}
q^\pr_{i_2,q_{i_1}}
q^{\pr\pr}_{i_2,q_{i_1}+q_{i_2}}
\ldots 
q^\pr_{i_k,q_{i_1}+\ldots+ q_{i_{k-1}}}
q^{\pr\pr}_{i_k,q_{i_1}+\ldots+q_{i_k}},\\
\pi_0(s_A(\Qbf^\pr,\Qbf^{\pr\pr},Q,\bar{0}))=
q^\pr_{j_1,0}
q^{\pr\pr}_{j_1,q_{j_1}}
q^\pr_{j_2,q_{j_1}}
q^{\pr\pr}_{j_2,q_{j_1}+q_{i_2}}
\ldots 
q^\pr_{j_l,q_{j_1}+\ldots+ q_{j_{l-1}}}
q^{\pr\pr}_{j_l,q_{j_1}+\ldots+q_{j_l}}.
%\label{eq:pi_0_s_A(pr,prpr,Q)}
\end{eqnarray*}
The original equation $E(X)\colon t(X)=s(X)$ defines the equation $t_B(X)=s_B(X)$ over $B$, where the terms $t_B(X),s_B(X)$ are defined by~(\ref{eq:t_B(X)},\ref{eq:s_B(X)}). By the definition of the set $T_<$, all terms 
\begin{eqnarray*}
0,&&\\
{x_{i_1}},&&{x_{j_1}}\\
{x_{i_1}}+{x_{i_2}},&&{x_{j_1}}+{x_{j_2}}\\
\ldots,&&\ldots \\
 {x_{i_1}+\ldots+ x_{i_{k-1}}},&&{x_{j_1}+\ldots+ x_{j_{l-1}}}\\
\end{eqnarray*}
belong to $T_<$. Thus, the definition of the points $\Qbf^\pr,\Qbf^{\pr\pr}$ gives:
\begin{eqnarray*}
%\label{eq:pi_0_t_A(pr,prpr,Q)}
\pi_0(t_A(\Qbf^\pr,\Qbf^{\pr\pr},Q,\bar{0}))=
p^\pr_{i_1,0}
p^{\pr\pr}_{i_1,p_{i_1}}
p^\pr_{i_2,p_{i_1}}
p^{\pr\pr}_{i_2,p_{i_1}+p_{i_2}}
\ldots 
p^\pr_{i_k,p_{i_1}+\ldots+ p_{i_{k-1}}}
p^{\pr\pr}_{i_k,p_{i_1}+\ldots+p_{i_k}}\\
=\pi_0(t_A(\Pbf^\pr,\Pbf^{\pr\pr},P,\bar{0})),\\
\pi_0(s_A(\Qbf^\pr,\Qbf^{\pr\pr},Q,\bar{0}))=
p^\pr_{j_1,0}
p^{\pr\pr}_{j_1,p_{j_1}}
p^\pr_{j_2,p_{j_1}}
p^{\pr\pr}_{j_2,p_{j_1}+p_{i_2}}
\ldots 
p^\pr_{j_l,p_{j_1}+\ldots+ p_{j_{l-1}}}
p^{\pr\pr}_{j_l,p{j_1}+\ldots+ p_{j_{l-1}}+p_{j_l}}\\
=\pi_0(s_A(\Pbf^\pr,\Pbf^{\pr\pr},P,\bar{0})).
%\label{eq:pi_0_s_A(pr,prpr,Q)}
\end{eqnarray*}

According to Lemma~\ref{l:to_first_projection}, the point $(\Pbf^\pr,\Pbf^{\pr\pr})$ does not satisfy the $0$-th projection of $E_A(X^\pr,X^{\pr\pr},P,\bar{0})$. In other words,
\[
\pi_0(t_A(\Qbf^\pr,\Qbf^{\pr\pr},Q,\bar{0}))=\pi_0(t_A(\Pbf^\pr,\Pbf^{\pr\pr},P,\bar{0}))\neq 
\pi_0(s_A(\Pbf^\pr,\Pbf^{\pr\pr},P,\bar{0}))=\pi_0(s_A(\Qbf^\pr,\Qbf^{\pr\pr},Q,\bar{0})).
\]

\end{proof}

% come back to normal variables

Let us come back to the original variables $X$. Using the variable substitution~(\ref{eq:new_variables}), we define a point $\Qbf=(\qbf_1,\ldots,\qbf_n)$, $\qbf_i=(q_{i,1},q_{i,2},\ldots)$ by
\[
\qbf_i\approx \qbf_i^\pr\s_{q_i}(\qbf^{\pr\pr}_i)
\]
or equivalently
\begin{equation}
q_{i,b}=q^\pr_{i,(b+1)}q^{\pr\pr}_{i,q_i+(b+1)}\; (1\leq i\leq n, b\in B).
\label{eq:Q_from_Q'Q''}
\end{equation}
%The sense of the index $(b+1)$ is the following. The points $\Qbf^\pr,\Qbf^{\pr\pr}\in \Pi_0^n$ were defined for the semigroup $C_0$ and, therefore, the first index of $\Qbf^\pr,\Qbf^{\pr\pr}$ is $0$. The point $\Qbf$ is defined for the semigroup $C$ and vectors $\qbf_i$ should begin with the index $1$. Thus,~(\ref{eq:Q_from_Q'Q''}) actually makes a one-step shift of the points $\Qbf^\pr,\Qbf^{\pr\pr}$.

The points $\Qbf,Q$ defined by~(\ref{eq:Q_from_Q'Q''},\ref{eq:Q}) satisfy the following:
\begin{enumerate}
\item $\Qbf\in\Pi^n$, since each $\qbf_i$ does not actually contain the unit $1\in A_1$;
\item $Q\in\V_B(S^\ast_B(X))$ and $\Qbf\in\V_\Pi(S^\ast_B(X,Q))$, i.e. $(\Qbf,Q)\in\V_C(S^\ast(X))$;
\item $\neg \pi_1(E_A(\Qbf,Q))$.
\end{enumerate}

\begin{lemma}
\label{l:Q_satisfies_S}
For the points $\Qbf,Q$ we have $(\Qbf,Q)\in\V_C(S(X))$.
%\begin{eqnarray}
%Q\in\V_B(S_B(X)),\\
%\Qbf\in\V_\Pi(S_B(X,Q)).
%\end{eqnarray}
\end{lemma}
\begin{proof}
First, we observe the following: the definition of the points $\Qbf^\pr,\Qbf^{\pr\pr}$ gives that the number of nonzero coordinates in $\qbf_i^\pr,\qbf^{\pr\pr}_i$ is at most $|T_<|$. Hence, the number of nonzero coordinates in $\qbf_i$ is not more than $|T_<|$.

Let $t(X)=s(X)$ be an arbitrary equation from $S\setminus S^\ast$. Since $S^\ast$ contains $\hat{S}$~(\ref{eq:S_hat}), the point $Q$ satisfies $t_B(X)=s_B(X)$. Let us consider the following shift terms
\begin{eqnarray*}
t_A(X,Q)=x_{i_1}\s_{q_{i_1}}(x_{i_2})\ldots \s_{q_{i_1}+\ldots+ q_{i_{k-1}}}(x_{i_k}),\\
s_A(X,Q)=x_{j_1}\s_{q_{j_1}}(x_{j_2})\ldots \s_{q_{j_1}+\ldots+ q_{j_{l-1}}}(x_{j_l}).
\end{eqnarray*}
By the choice of the subsystem $S^\ast$, we have $k,l>|T_<|$. According to Lemma~\ref{l:long_terms}, both values $t_A(\Qbf,Q)$, $s_A(\Qbf,Q)$ equal $(0,0,\ldots,)\in \Pi$. Thus, $\Qbf\in \V_\Pi(t_A(X,Q)=s_A(X,Q))$ and $\Qbf\in \V_\Pi(S_A)$.
\end{proof}

Lemma~\ref{l:Q_satisfies_S} provides
\[
(\Qbf,Q)\in \V_C(S(X))\setminus \V_C(E)
\]
that contradicts the inclusion~(\ref{eq:inclusion1}). Thus, {\bf Theorem~B} is proved.


\bigskip
\begin{thebibliography}{1}

\bibitem{BMRom}
Baumslag G., Myasnikov A.,  Romankov V. {\it Two theorems about equationally
Noetherian groups}. J. Algebra, 194 (1997), 654--664.


\bibitem{DMR2}
E. Yu. Daniyarova, A. G. Myasnikov, V. N. Remeslennikov, Algebraic geometry over algebraic structures. II. Foundations, J. Math. Sci., 185:3 (2012), 389--416

\bibitem{DMR3}
E. Daniyarova, A. Miasnikov, V. Remeslennikov, Algebraic geometry over algebraic structures III: Equationally Noetherian property and compactness, South. Asian Bull. Math., 35:1 (2011), 35--68

\bibitem{shevl1}
P. Morar, A. Shevlyakov, Algebraic Geometry over the Additive Monoid of Natural Numbers: Systems of Coefficient Free Equations, Combinatorial and Geometric Group Theory: Dortmund and Carleton Conferences, Birkhauser, 2010, 261--278

\bibitem{plotkin_monster}
B. I. Plotkin, Algebras with the Same (Algebraic) Geometry, Proc. Steklov Inst. Math., 242 (2003), 165--196

\bibitem{plotkin}
B. I. Plotkin, Problems in algebra inspired by universal algebraic geometry, J. Math. Sci., 139:4 (2006), 6780--6791

\bibitem{plotkin_lectures}
B. I. Plotkin, Seven lectures on universal algebraic geometry, Comtemp. Math., 726 (2019), 143--216


\bibitem{shevl2}
A. Shevlyakov, Algebraic geometry over the monoid of natural numbers. Irreducible algebraic sets, Proceedings of the Institute of Mathematics and Mechanics (Trudy Instituta Matematiki I Mekhaniki), 16:2 (2010), 258--269 

\bibitem{shevl3}
A. Shevlyakov, Algebraic geometry over the additive monoid of natural numbers: The classifcation of coordinate monoids 2(1), p. 91-111,2010, Groups, Complexity and Cryptology, 2:1 (2010), 91--111

\bibitem{shevl_seams}
A. Shevlyakov, Commutative Idempotent Semigroups at the Service of Universal Algebraic Geometry, Southeast Asian Bulletin of Mathematics, 35 (2011), 111--136

\bibitem{shevl_shah}
A. Shevlyakov, M. Shahryary, Direct products, varieties, and compactness conditions, Groups Complexity Cryptology, 9:2 (2017), 159--166

\bibitem{shevl_aut}
A. Shevlyakov, On group automorphisms in universal algebraic geometry, Groups, Complexity, Cryptology, 11:2 (2019), 115--122

\end{thebibliography}
\end{document}